\documentclass[12pt]{amsart}

\newtheorem{theorem}{Theorem}[section]

\theoremstyle{plain}
\newtheorem{definition}[theorem]{Definition}

\newtheorem{question}[theorem]{Question}

\theoremstyle{definition}

\theoremstyle{remark}

\numberwithin{equation}{section}

\usepackage{fancyhdr}
\usepackage{amscd}
\usepackage{amsmath}
\usepackage{amsthm}
\usepackage{amssymb}

\begin{document}

\title[On Kernel of the regulator map]{On Kernel of the regulator map}

\author{Sen Yang}
\address{Yau Mathematical Sciences Center, Tsinghua University
\\
Beijing, China}
\email{syang@math.tsinghua.edu.cn; senyangmath@gmail.com}

\subjclass[2010]{14C25}
\date{}

\begin{abstract}
By using the infinitesimal methods due to Bloch, Green and Griffiths in \cite{Bloch1, GGtangentspace}, we construct an infinitesimal form of the regulator map and verify that its kernel 
is $\Omega_{\mathbb{C}/ \mathbb{Q}}^{1}$, which suggests that Question ~\ref{question: main} seems reasonable at the infinitesimal level.
\end{abstract}

\maketitle

\tableofcontents

\section{Background and question}
\label{Background and question}
Let $X$ be a smooth projective curve over the complex number field $\mathbb{C}$. In 1970s, Bloch constructed the regulator map R: $ K_{2}(X) \to H^{1}(X, \mathbb{C}^{\ast})$ in several ways. Later, Deligne found a different construction by considering $H^{1}(X, \mathbb{C}^{\ast})$ as the group of line bundles with connections. We recall his construction very briefly as follows.

For $x$ a point on $X$, we use $i_{x}$ to denote the inclusion $ x \to X$.
The flasque BGQ resolution of $K_{2}(O_{X})$
{\footnotesize
\begin{align*}
 0   \rightarrow K_{2}(O_{X})   \to K_{2}(\mathbb{C}(X)) \to \bigoplus_{x \in X^{(1)}}i_{x, \ast}K_{1}(\mathbb{C}(x)) \to 0,
\end{align*}
} 
shows that $H^{0}(K_{2}(O_{X}))$ can be computed as $\mathrm{Ker}\{ K_{2}(\mathbb{C}(X))  \to \bigoplus\limits_{x \in X^{(1)}} K_{1}(\mathbb{C}(x)) \}$. So we have the exact sequence of groups
{\footnotesize
\begin{align*}
 0   \to H^{0}(K_{2}(O_{X})) \to K_{2}(\mathbb{C}(X)) \to \bigoplus_{x \in X^{(1)}}K_{1}(\mathbb{C}(x)).
\end{align*}
} 

It's known that there exists the following Gysin exact sequence in topology,
{\footnotesize
\begin{align*}
 0 \to  H^{1}(X, \mathbb{C}^{\ast}) \to H^{1}(\mathbb{C}(X), \mathbb{C}^{\ast}) \to \bigoplus\limits_{x \in X^{(1)}} \mathbb{C}^{\ast}, 
\end{align*}
} 
where $ H^{1}(\mathbb{C}(X), \mathbb{C}^{\ast}) = \underrightarrow{\mathrm{lim}}H^{1}(X-S, \mathbb{C}^{\ast})$ and $S$ is finite points on $X$.

The main ingredient to construct the regulator map R: $H^{0}(K_{2}(O_{X})) \to H^{1}(X, \mathbb{C}^{\ast})$  is the following commutative diagram
{\footnotesize
 \begin{equation}
  \begin{CD}
       0 @>>> H^{0}(K_{2}(O_{X}))  @>>> K_{2}(\mathbb{C}(X)) @>>> \bigoplus\limits_{x \in X^{(1)}} K_{1}(\mathbb{C}(x)) \\
     @VVV @V \mathrm{R} VV @V \mathrm{R} VV  @V \cong VV\\
     0 @>>>  H^{1}(X, \mathbb{C}^{\ast}) @>>> H^{1}(\mathbb{C}(X), \mathbb{C}^{\ast}) @>>> \bigoplus\limits_{x \in X^{(1)}} \mathbb{C}^{\ast}.
  \end{CD}
\end{equation}
}
That is, one constructs a map R: $ K_{2}(\mathbb{C}(X)) \to H^{1}(\mathbb{C}(X), \mathbb{C}^{\ast})$ and use it to deduce the regulator map R: $ H^{0}(K_{2}(O_{X})) \to H^{1}(X, \mathbb{C}^{\ast})$. We refer the readers to \cite{Bloch1} and Section 6 in \cite{Hain} for more details.

This regulator map has nice motivic features and is related with a general program of Bloch-Beilinson conjecture. In this short note, we focus on the following question, see Section 2 in \cite{GGregulator} for related discussion. To fix notations, for any abelian group $M$, $M_{\mathbb{Q}}$ denotes the image of $M$ in $M \otimes_{\mathbb{Z}} \mathbb{Q}$ in the following.

\begin{question} [Conjecture 2.4 in \cite{GGregulator}]  \label{question: main}
Let $\mathrm{R}:$ $ H^{0}(K_{2}(O_{X})) \to H^{1}(X, \mathbb{C}^{\ast})$ be the regulator map, 
then $\mathrm{Ker(R)}_{\mathbb{Q}} = K_{2}(\mathbb{C})_{\mathbb{Q}}$.
\end{question}
This question is very difficult to approach, though it has very simple form. For $X=\mathrm{P}^{1}$, this conjecture has been verified by Kerr \cite{Kerr}. 

\textbf{Acknowledgements}
The author is very grateful to Phillip Griffiths, James Lewis and Kefeng Liu for discussions, and to Spencer Bloch and Matt Kerr for comments on previous version. He also thanks his colleagues Eduard Looijenga and Thomas Farrell for explaining questions in \cite{Bloch1}.

Many thanks to anonymous referee(s) for careful reading and professional suggestions, which improves this note a lot. 


\section{Main results}
\label{Main results}
In this section, we shall define an infinitesimal form of the regulator map R: $ H^{0}(K_{2}(O_{X})) \to H^{1}(X, \mathbb{C}^{\ast})$ and verify its kernel is $\Omega_{\mathbb{C} / \mathbb{Q}}^{1}$. Our approach is inspired by the following result due to Green and Griffiths:
 \begin{theorem} [Page 74 and page 125 in \cite{GGtangentspace}]  \label{theorem:GGtangentBGQ}
Let $X$ be a smooth projective curve over $\mathbb{C}$, the Cousin flasque resolution of $\Omega_{X/ \mathbb{Q}}^{1}$
  \[
  0 \to \Omega_{X/ \mathbb{Q}}^{1} \to \Omega_{\mathbb{C}(X)/ \mathbb{Q}}^{1} \xrightarrow{\rho} \bigoplus\limits_{x \in X^{(1)}} i_{x,\ast}H_{x}^{1}(\Omega_{X/\mathbb{Q}}^{1}) \to 0,
  \]
   is the tangent sequence to BGQ flasque resolution of the sheaf $K_{2}(O_{X})$
   {\footnotesize
\begin{align*}
 0   \rightarrow K_{2}(O_{X})   \to K_{2}(\mathbb{C}(X)) \rightarrow \bigoplus_{x \in X^{(1)}} i_{x,\ast} K_{1}(\mathbb{C}(x)) \to 0, 
\end{align*}
} 
where the map $\rho$  is known to take principal parts.   
\end{theorem}

It follows that $H^{0}(\Omega_{X/ \mathbb{Q}}^{1})$ can be computed as $\mathrm{Ker}\{\Omega_{\mathbb{C}(X)/ \mathbb{Q}}^{1}   \xrightarrow{\rho} \bigoplus\limits_{x \in X^{(1)}} H_{x}^{1}(\Omega_{X/\mathbb{Q}}^{1}) \}$. So we have the exact sequence of groups
{\footnotesize
\begin{align*}
 0   \to H^{0}(\Omega_{X/ \mathbb{Q}}^{1}) \to  \Omega_{\mathbb{C}(X)/ \mathbb{Q}}^{1}   \xrightarrow{\rho} \bigoplus\limits_{x \in X^{(1)}} H_{x}^{1}(\Omega_{X/\mathbb{Q}}^{1}).
\end{align*}
}

\begin{definition} [ page 71 and page 125 in \cite{GGtangentspace} ]  \label{definition: Residue}
 For $X$ a smooth projective curve over $\mathbb{C}$ and $x$ a point on $X$, there exists a residue map
 \[
  \mathrm{Res}: H_{x}^{1}(\Omega_{X/\mathbb{Q}}^{1}) \to \mathbb{C},
 \]
 which is defined as follows:
 
Using $\Omega_{O_{X,x} /\mathbb{Q}}^{1}(nx)$ to denote the absolute 1-forms with poles of order at most $n$ at $x$,  we define $\mathrm{Res}_{x}$ as the following composition:
\[
\Omega_{O_{X,x}/\mathbb{Q}}^{1}(nx) \longrightarrow \Omega_{O_{X,x}/\mathbb{C}}^{1}(nx) \xrightarrow{\mathrm{Res}} \mathbb{C}.
\]
If $\xi$ is the local uniformizer centered at $x$, an element of $H_{x}^{1}(\Omega_{X/\mathbb{Q}}^{1})$  is represented by the following diagram
\begin{equation}
\begin{cases}
 \begin{CD}
   O_{X,x} @>\xi^{k}>> O_{X,x} @>>>  O_{X,x}/(\xi^{k})@>>> 0  \\
   O_{X,x} @>\psi>> \Omega_{O_{X,x}/ \mathbb{Q}}^{1}.
 \end{CD}
\end{cases}
\end{equation}
For such an element, we define $\mathrm{Res}_{x}(\dfrac{ \psi}{\xi^{k}}) \in \mathbb{C}$.
 
\end{definition}

It is known that the tangent space to $\mathbb{C}^{\ast}$, which  is defined to be the kernel of the natural projection:
\[
  \mathbb{C}[\varepsilon]^{\ast}  \xrightarrow{\varepsilon =0}  \mathbb{C}^{\ast},
\]     
can be identified with $\mathbb{C}$ and the tangent map tan: $\mathbb{C}[\varepsilon]^{\ast} \to \mathbb{C}$
is given by $z_{0} + z_{1}\varepsilon  \to \dfrac{z_{1}}{z_{0}}$. 
This tangent map further induces a map between cohomology groups tan: $ H^{1}(X,  \mathbb{C}[\varepsilon]^{\ast})  \to H^{1}(X, \mathbb{C})$. With this interpretation, one can consider $H^{1}(X, \mathbb{C})$ as the tangent space to $H^{1}(X, \mathbb{C}^{\ast})$(this is used in \cite{Bloch1}).

There exists  the following Gysin exact sequence in topology:
\[
0 \to H^{1}(X, \mathbb{C}) \to H^{1}(\mathbb{C}(X), \mathbb{C}) \to \bigoplus\limits_{x \in X^{(1)}} \mathbb{C},
\]
e.g., see page 54-55 in \cite{Green}. The boundary map $H^{1}(\mathbb{C}(X), \mathbb{C}) \to \bigoplus\limits_{x \in X^{(1)}} \mathbb{C}$ can be described via Hodge theory as follows. Let $D=\{p_{1}, \cdots, p_{n}\}$ be finite points on $X$ and let $U$ be the open complement, $U = X- D$. Let $i_{D}: D \to X$ denote the inclusion, the residue map Res: $ \Omega^{\bullet}_{X}(\mathrm{log}D) \to i_{D,\ast}\Omega^{\bullet -1}_{D} $ induces 
      Res: $ \mathbb{H}^{1}(\Omega^{\bullet}_{X}(\mathrm{log}D)) \to \mathbb{H}^{0}(\Omega^{\bullet}_{D}) $. This gives the map Res: $ H^{1}(U, \mathbb{C}) \to \bigoplus\limits_{i=1, \cdots, n}  \mathbb{C}$, by using the identifications  
     $\mathbb{H}^{1}(\Omega^{\bullet}_{X}(\mathrm{log}D)) \cong H^{1}(U, \mathbb{C})$ and $\mathbb{H}^{0}(\Omega^{\bullet}_{D}) = H^{0}(D, \mathbb{C}) \cong \bigoplus\limits_{i=1, \cdots, n}  \mathbb{C}$.

The following theorem is an infinitesimal form of diagram (1.1):
\begin{theorem} \label{theorem: tangReg}
There exists the following commutative diagram
{\footnotesize
\begin{equation}
  \begin{CD}
    0 @>>> H^{0}(\Omega_{X/ \mathbb{Q}}^{1}) @>>>  \Omega_{\mathbb{C}(X)/ \mathbb{Q}}^{1} @>\rho>> \bigoplus\limits_{x \in X^{(1)}} H_{x}^{1}(\Omega_{X/\mathbb{Q}}^{1})  \\
   @VVV  @V \mathrm{R'} VV @V \mathrm{R'} VV  @V \mathrm{Res} VV\\
     0 @>>>  H^{1}(X, \mathbb{C}) @>>> H^{1}(\mathbb{C}(X), \mathbb{C}) @>\mathrm{Res}>> \bigoplus\limits_{x \in X^{(1)}} \mathbb{C},
  \end{CD}
\end{equation}
}
where the map $\mathrm{R'}$'s are the natural maps sending $d_{/ \mathbb{Q}}f$ to $d_{/ \mathbb{C}}f$.
\end{theorem}

\begin{proof}

The map $\mathrm{R'}$: $\Omega_{\mathbb{C}(X)/ \mathbb{Q}}^{1} \to H^{1}(\mathbb{C}(X), \mathbb{C})$ can be described as follows. Let $U$ be open affine in $X$, $H^{1}(U, \mathbb{C})$ can be computed as $\Gamma(U,\Omega_{U/\mathbb{C}})/ d_{/ \mathbb{C}} \Gamma(U, O_{U})$. Given any element $\alpha \in \Omega_{U/ \mathbb{Q}}^{1}$, its image $[\alpha]$ in $\Omega_{U/ \mathbb{C}}^{1}$ defines an element in $H^{1}(U, \mathbb{C})$.

To check the commutativity of the right square,  
working locally in a Zariski open affine neighborhood U, we can write an element $ \beta \in \Omega_{\mathbb{C}(X) / \mathbb{Q}}^{1}$ as 
\[
 \beta = \dfrac{h \ d_{/ \mathbb{Q}}g}{f^{l_{1}}_{1}\dots f^{l_{k}}_{k}},
\]
where $f_{1}, \dots, f_{k}, h \in \Gamma(U,O_{U})$ are relatively prime and $f_{i}'s$ are irreducible.

The following diagram is commutative:
\[
  \begin{CD}
     \dfrac{h \ d_{/ \mathbb{Q}}g}{f^{l_{1}}_{1}\dots f^{l_{k}}_{k}}  @>\rho>> \sum_{i} \dfrac{h \ d_{/ \mathbb{Q}}g}{f^{l_{1}}_{1}\dots \hat{f}^{l_{i}}_{i} \dots f^{l_{k}}_{k}} \\
  @V \mathrm{R'}VV   @V \mathrm{Res} VV\\
     \dfrac{h \ d_{/ \mathbb{C}}g}{f^{l_{1}}_{1}\dots f^{l_{k}}_{k}}   @>\mathrm{Res}>> \sum_{i} \mathrm{Res}_{x_{i}}(\dfrac{h \ d_{/ \mathbb{C}}g}{f^{l_{1}}_{1}\dots f^{l_{k}}_{k}} ),
  \end{CD}
\]
where $x_{i}=\{f_{i} = 0 \}$ and $\hat{f}^{l_{i}}_{i}$ means to omit the $i^{th}$ term. 

The map $\mathrm{R'}$: $\Omega_{\mathbb{C}(X)/ \mathbb{Q}}^{1} \to H^{1}(\mathbb{C}(X), \mathbb{C})$ induces $\mathrm{R'}$: $ H^{0}(\Omega_{X/ \mathbb{Q}}^{1})  \to H^{1}(X, \mathbb{C})$.
\end{proof}

Let $\{f_{0},g_{0} \} \in H^{0}(K_{2}(O_{X})) $  and
let $(N,\bigtriangledown)$ denote the bundle with connection $\bigtriangledown$, as recalled on page 4 in \cite{Bloch1}. There exists the following commutative diagram:
\[
 \begin{CD}
    \{f_{0},g_{0} \}   @<\varepsilon=0<<  \{f_{0}+\varepsilon f_{1},g_{0}+\varepsilon g_{1} \}  @>\mathrm{tan}>>  \dfrac{f_{1}}{f_{0}}\dfrac{d_{/ \mathbb{Q}}g_{0}}{g_{0}} -  \dfrac{g_{1}}{g_{0}}\dfrac{d_{/ \mathbb{Q}}f_{0}}{f_{0}}\\
   @V \mathrm{R} VV    @VVV  @V \mathrm{R'} VV \\
   \{f_{0},g_{0} \}^{\ast} (N,\bigtriangledown) @<<\varepsilon=0<   \{f_{0}+\varepsilon f_{1},g_{0}+\varepsilon g_{1} \}^{\ast} (N,\bigtriangledown)   @>\mathrm{tan}>> \dfrac{f_{1}}{f_{0}}\dfrac{d_{/ \mathbb{C}}g_{0}}{g_{0}} -  \dfrac{g_{1}}{g_{0}}\dfrac{d_{/ \mathbb{C}}f_{0}}{f_{0}}.
 \end{CD}
\]

The commutativity of left square is trivial. To check the right one, since $\{f_{0}+\varepsilon f_{1},g_{0}+\varepsilon g_{1} \} = \{f_{0},g_{0} \} \{f_{0}, 1+\varepsilon \dfrac{g_{1}}{g_{0}} \} \{1+ \varepsilon\dfrac{f_{1}}{f_{0}}, g_{0} \} \{1+ \varepsilon\dfrac{f_{1}}{f_{0}},  1+\varepsilon \dfrac{g_{1}}{g_{0}} \}$, we reduce to considering  $\{1+\varepsilon f_{1}, g_{0} \}$ which is obvious:
\[
 \begin{CD}
   \{1+\varepsilon f_{1},g_{0} \}  @>\mathrm{tan}>>  f_{1}\dfrac{d_{/ \mathbb{Q}}g_{0}}{g_{0}} \\
     @VVV  @V \mathrm{R'} VV \\
    \{1+\varepsilon f_{1},g_{0} \}^{\ast} (N,\bigtriangledown)   @>\mathrm{tan}>> f_{1}\dfrac{d_{/ \mathbb{C}}g_{0}}{g_{0}},
 \end{CD}
\]
where the up tan map is well-known and the down tan map is the formula (2.12) on page 14 in [1].

In this sense, we consider the map $\mathrm{R'}:$ $ H^{0}(\Omega_{X/ \mathbb{Q}}^{1})  \to H^{1}(X, \mathbb{C})$ as the infinitesimal form of the regulator map R: $H^{0}(K_{2}(O_{X})) \to H^{1}(X, \mathbb{C}^{\ast})$  and computer the kernel of  $\mathrm{R'}$.

Since $H^{1}(X, \mathbb{C})$ has Hodge decomposition $H^{1}(X, \mathbb{C}) \cong H^{0}(\Omega_{X/\mathbb{C}}^{1}) \oplus H^{1}(O_{X})$ and the map $\mathrm{R'}$: $H^{0}(\Omega_{X/ \mathbb{Q}}^{1}) \to H^{1}(X, \mathbb{C})$ naturally maps $d_{/ \mathbb{Q}}f$ to $d_{/ \mathbb{C}}f$, so $\mathrm{R'}$
is the composition $H^{0}(\Omega_{X/ \mathbb{Q}}^{1}) \to  H^{0}(\Omega_{X/\mathbb{C}}^{1}) \hookrightarrow H^{1}(X, \mathbb{C})$. Hence $\mathrm{Ker(R')} = \mathrm{Ker}\{ H^{0}(\Omega_{X/ \mathbb{Q}}^{1}) \to  H^{0}(\Omega_{X/\mathbb{C}}^{1}) \}$.

\begin{theorem} \label{theorem: main}
$\mathrm{Ker(R')} = \Omega_{\mathbb{C}/ \mathbb{Q}}^{1}$.

\end{theorem}

\begin{proof}
There is a natural short exact sequence of sheaves
\[
0 \to \Omega_{\mathbb{C}/\mathbb{Q}}^{1}\otimes_{\mathbb{C}} O_{X} \to \Omega_{X/\mathbb{Q}}^{1} \to \Omega_{X/\mathbb{C}}^{1} \to 0.
\]

The associated long exact sequence is of the form
\[
0 \to H^{0}(\Omega_{\mathbb{C}/\mathbb{Q}}^{1}\otimes_{\mathbb{C}} O_{X}) \to H^{0}(\Omega_{X/\mathbb{Q}}^{1}) \to H^{0}(\Omega_{X/\mathbb{C}}^{1}) \to H^{1}(\Omega_{\mathbb{C}/\mathbb{Q}}^{1}\otimes_{\mathbb{C}} O_{X}) \to \cdots.
\]
So the kernel of $H^{0}(\Omega_{X/\mathbb{Q}}^{1}) \to H^{0}(\Omega_{X/\mathbb{C}}^{1})$
can be identified with $H^{0}(\Omega_{\mathbb{C}/\mathbb{Q}}^{1} \otimes_{\mathbb{C}} O_{X}) \cong H^{0}(O_{X}) \otimes\Omega_{\mathbb{C}/ \mathbb{Q}}^{1} \cong \mathbb{C} \otimes \Omega_{\mathbb{C}/ \mathbb{Q}}^{1} \cong \Omega_{\mathbb{C}/ \mathbb{Q}}^{1}$.

\end{proof}

Since the tangent space to $K_{2}(\mathbb{C})$ is $ \Omega_{\mathbb{C}/ \mathbb{Q}}^{1}$, this theorem suggests hat Question ~\ref{question: main} seems reasonable at the infinitesimal level.

\end{document}